\documentclass[11pt]{amsart}
\usepackage{amsmath}
\usepackage{extarrows}
\usepackage{amssymb,amsfonts}
\usepackage[all,arc]{xy}
\usepackage{enumerate}
\usepackage{enumitem}
\usepackage{mathrsfs}
\usepackage[hyperindex]{hyperref}
\usepackage{color}
\usepackage{mathtools}
\usepackage[utf8]{inputenc}
\usepackage{psfrag,euscript}
\usepackage{graphicx}
\usepackage{tikz}
\usetikzlibrary{arrows}
\usepackage{verbatim}
\usepackage{tikz-cd}
\usetikzlibrary{3d,positioning}
\usepackage{rotating}
\usepackage{tabularx}
\usepackage{multirow}
\usepackage{babel} 
\theoremstyle{definition}

\newtheorem{teorema}{Theorem}[section]

\newtheorem{proposicao}[teorema]{Proposition}

\newcommand{\downarrowtail}{\vcenter{\hbox{\rotatebox{90}{$\leftarrowtail$}}}}

\usepackage[nodisplayskipstretch]{setspace}
\theoremstyle{definition}
\newtheorem{definicao}[teorema]{Definition}

\newtheorem{exemplo}[teorema]{Example}
\newtheorem{examples}[teorema]{Examples}

\newtheorem{obs}[teorema]{Remark}

\theoremstyle{remark}

\makeatletter

\title[Manin triples on multiplicative Courant algebroids]{{\large M}anin triples on multiplicative {\large C}ourant algebroids}

\author{{\small A}na {\small C}arolina {\small M}ançur \\ 
{\small \href{mailto:carol.mancu@gmail.com}{\texttt{\lowercase{carol.mancu@gmail.com}}}}
}

\begin{document}

\begin{abstract} 

We extend the characterization of Lie bialgebroids via Manin triples to the context of double structures over Lie groupoids. We consider Lie bialgebroid groupoids, given by LA-groupoids in duality, and establish their correspondence with multiplicative Manin triples, i.e., CA-groupoids equipped with a pair of complementary multiplicative Dirac structures. As an application, we give a new viewpoint to the co-quadratic Lie algebroids of Lang, Qiao, and Yin \cite{lang2021lie} and the Manin triple description of Lie bialgebroid crossed modules.

\end{abstract}

\maketitle
\pagestyle{plain}

\tableofcontents


\section{Introduction}

In the theory of Poisson Lie groups \cite{drinfel1990hamiltonian}, a classical result proved by Drinfeld in \cite{drinfel1988quantum} establishes an equivalence between Lie bialgebras and Manin triples, i.e., quadratic Lie algebras equipped with a pair of transverse Lagrangian Lie subalgebras. An important generalization of this result, due to Liu-Weinstein-Xu \cite{liu1997manin}, is obtained when Lie algebras are replaced by Lie algebroids; in this setting, to obtain the Manin triples corresponding to Lie bialgebroids, one needs to replace quadratic Lie algebras with Courant algebroids, and Manin triples are defined by pairs of transverse Dirac structures there in. 
This paper concerns a further extension of this correspondence to the realm of ``double structures over Lie groupoids'', in the sense explained below.

\vspace{1mm}

Just as Lie bialgebroids arise as infinitesimal counterparts of Poisson groupoids \cite{mackenzie1994lie}, the (first order) differentiation of a Poisson double groupoid \cite{mackenzie1999symplectic} gives rise to ``Lie bialgebroid groupoids'', see \cite{bursztyn2021poisson}; those consist of two LA-groupoids in duality, depicted by
$$
\begin{tikzcd}[
		column sep={2em,between origins},
		row sep={1.3em,between origins},]
		\Gamma & \rightrightarrows & H  \\
		\Downarrow & & \Downarrow \\
		G & \rightrightarrows & M,
\end{tikzcd}
\quad
\begin{tikzcd}[
		column sep={2em,between origins},
		row sep={1.3em,between origins},]
		\,\, \Gamma^{*_G} \,\,\, & \rightrightarrows & \,\, C^*  \\
		\Downarrow \,\,\, & & \Downarrow \\
		G \,\,\, & \rightrightarrows & M,
\end{tikzcd}
$$
such that the Lie algebroids $\Gamma \Rightarrow G$ and $\Gamma^{*_G} \Rightarrow G$ form a Lie bialgebroid.

\vspace{1mm}

In order to obtain a description of Lie bialgebroid groupoids via Manin triples, the key issue is identifying the double nature of their Drinfeld doubles. In the picture above,  
$\Gamma \oplus_G \Gamma^{*_G} \rightarrowtail G$ is a Courant algebroid which carries an additional Lie groupoid structure $\Gamma \oplus_G \Gamma^{*_G} \rightrightarrows H\oplus C^*$; our central observation is that these two structures are compatible in the sense that they form a {\em CA-groupoid} (also known as {\em multiplicative Courant algebroid}), as defined in \cite{li2012courant}. With that, we establish in Theorem \ref{teorema1} the equivalence between Lie bialgebroids groupoids and CA-groupoids equipped with two complementary and multiplicative Dirac structures. In \cite[\S 4]{alvarez2024homological}, the authors investigate, from a supergeometric perspective, a more general version of CA-groupoids, referred to as quasi CA-groupoids, and explore this object in double contexts. However, they do not translate the definitions and results to the classical setting.

\vspace{2mm}

{\footnotesize
\bgroup
\def\arraystretch{1.6}
\begin{center}
\begin{tabular}{ |c|c|c| }
\hline
{\bf Global object} & {\bf (First-order) Infinitesimal object}  & {\bf Drinfeld double} \\ \hline
Poisson Lie group & Lie bialgebra & quadratic Lie algebra
\\ \hline
Poisson groupoid & Lie bialgebroid  & Courant algebroid
\\ \hline
Poisson double groupoid & Lie bialgebroid groupoid  & CA-groupoid
\\ \hline
\end{tabular}
\end{center}
\egroup
}

\vspace{2mm}

In the last section of the paper, we focus on structures over the  unit groupoid $M \rightrightarrows M$. In this case, Lie bialgebroid groupoids are the same as the {\em Lie bialgebroid crossed modules} of \cite{lang2021lie}. 
We show in Theorem \ref{teorema_co-quadLA} that CA-groupoids over the unit groupoid are equivalent to  {\em co-quadratic Lie algebroids}, in the sense of \cite{lang2021lie}, in such a way that multiplicative Manin triples correspond to co-quadratic Manin triples. This perspective recovers \cite[Theorem 3.7]{lang2021lie}, which states that Lie bialgebroid crossed modules correspond to co-quadratic Manin triples, as a special case of Theorem \ref{teorema1}. When $M$ is a point, one recovers the characterization of Lie 2-bialgebras via Manin triples considered in \cite{bai2013lie}. 

\subsection*{Notations} We use the notation $G \rightrightarrows M$ for Lie groupoids, $A \Rightarrow M$ for Lie algebroids, and $E \rightarrowtail M$ for Courant algebroids.

\subsection*{Acknowledgments} The author thanks H. Bursztyn for his continuous support and valuable suggestions; J. P. Ayala for the fruitful conversations; and CAPES and IMPA for their financial support.


\section{Review of Manin triples for Lie bialgebroids}\label{section2}

In what follows, we will frequently use the characterization of Lie algebroid structures on a vector bundle $A\to M$ by means of linear Poisson structures on the total space of the dual vector bundle $A^*\to M$, see \cite[\S 10.3]{mackenzie2005general}.

\subsection{Lie bialgebroids}

Let $A \Rightarrow M$ be a Lie algebroid and suppose that its dual bundle $A^* \rightarrow M$ also has a Lie algebroid structure. We say that $(A, A^*)$ is a \textbf{Lie bialgebroid} if the Lie algebroid differential $d_A : \Gamma(A^*) \rightarrow \Gamma(\wedge^2 A^*)$ is a derivation of the Schouten bracket $[\, , \,]_{A^*}$ on $\Gamma(A^*)$ in the sense that
	$$
	d_{A}[\alpha, \beta]_{A^*} = [d_A\alpha, \beta]_{A^*} + [\alpha, d_A\beta]_{A^*},
	$$
	for all $\alpha, \beta \in \Gamma(A^*)$. 
 An important case of Lie bialgebroids is when $M$ is a point, called a \textbf{Lie bialgebra}.

\vspace{1mm}

\begin{exemplo}\label{exemplosbialgebroids}
    Given a Poisson manifold $(M, \pi)$, the pair $(TM, T^*M)$ is a Lie bialgebroid, where $TM \Rightarrow M$ is the tangent Lie algebroid and the Lie algebroid structure on $T^*M \Rightarrow M$ is given by $\pi$.
\end{exemplo}

\begin{exemplo}\label{exbialgbdr-matrix}
    Let $A \Rightarrow M$ be a Lie algebroid with bracket $[ \,\, ,\,]$ and anchor $\rho : A \rightarrow TM$. Given $\Lambda \in \Gamma(\wedge^2 A)$ such that $[[\Lambda, \Lambda], X] = 0$ for all $X \in \Gamma(A)$, the bracket
    $$
    [\alpha, \beta]_{\Lambda} := \mathcal{L}_{\Lambda^{\sharp} \alpha} \beta - \mathcal{L}_{\Lambda^{\sharp} \beta} \alpha - d_A(\Lambda(\alpha, \beta))
    $$
    on $\Gamma(A^*)$ and the anchor $\rho_* := \rho \circ \Lambda^{\sharp}$ make $A^* \Rightarrow M$ into a Lie algebroid, and $(A, A^*)$ forms a Lie bialgebroid known as \textit{exact Lie bialgebroid} \cite{kosmann1995exact, liu1996exact}. The element $\Lambda$ is called an \textit{r-matrix} on $A$.
\end{exemplo}

\subsection{Manin triples for Lie bialgebroids}\label{subsec:LWX}

We recall the definition of Courant algebroid \cite{liu1997manin} using non-skew-symmetric brackets, as in \cite{roytenberg1999courant}.

\vspace{1mm}

A \textbf{Courant algebroid} over a manifold $M$ is a vector bundle $\mathbb{E} \rightarrow M$, together with a bundle map $\rho : \mathbb{E} \rightarrow TM$, called \textit{anchor}, a non-degenerate symmetric bilinear form $\langle \,\, , \, \rangle$ on fibers, and a bracket $[\![ \,\, , \, ]\!]$ on $\Gamma(\mathbb{E})$, satisfying 
\begin{enumerate}[left=10pt, itemsep=0.3em]
	\item[(C$1$)] $[\![ e_1 , [\![ e_2 , e_3 ]\!] ]\!] = [\![ [\![ e_1 , e_2 ]\!] , e_3 ]\!] + [\![ e_2 , [\![ e_1 , e_3 ]\!] ]\!]$;
	\item[(C$2$)] $\rho(e_1) \langle e_2 , e_3 \rangle = \langle [\![ e_1 , e_2 ]\!] , e_3 \rangle + \langle e_2 , [\![ e_1 , e_3 ]\!] \rangle$;
	\item[(C$3$)] $[\![ e_1 , e_2 ]\!] + [\![ e_2 , e_1 ]\!] = 2 \rho^*(d \langle e_1, e_2 \rangle)$;
	\item[(C$4$)] $[\![ e_1 , f e_2 ]\!] = f [\![ e_1 , e_2 ]\!] + \rho(e_1)(f)e_2$;
\end{enumerate}
for $e_1, e_2, e_3 \in \Gamma(\mathbb{E})$ and $f \in C^{\infty}(M)$, where $\rho^* : T^*M \rightarrow \mathbb{E}^* \simeq \mathbb{E}$ is the dual map to $\rho$. We denote by $\overline{\mathbb{E}}$ the Courant algebroid with the same anchor and bracket of $\mathbb{E}$, but pairing $- \langle \,\, , \, \rangle$.

\vspace{1mm}

\begin{exemplo}\label{CAoverpoint}
    A Lie algebra $\mathfrak{g}$ together with an symmetric, nondegenerate pairing that is invariant,  i.e. $\langle [X,Y], Z\rangle + \langle Y, [X,Z]\rangle = 0$ for all $X, Y , Z \in \mathfrak{g}$, is called a \textbf{quadratic Lie algebra} and corresponds to a Courant algebroid over a point.
\end{exemplo}

\begin{exemplo}\label{standardCA}
The \textit{standard Courant algebroid} over a manifold $M$ is $\mathbb{T}M := TM \oplus T^*M$ with anchor map given by the projection on the first factor, the bilinear form as
$$
\langle (X, \alpha), (Y, \beta) \rangle := \frac{1}{2} (\beta(X) + \alpha(Y)),
$$
and the Courant bracket given by
$$
[\![ (X, \alpha) , (Y, \beta) ]\!] := ([X, Y], \mathcal{L}_X\beta - \iota_Yd\alpha),
$$
for vector fields $X, Y \in \mathfrak{X}(M)$ and 1-forms $\alpha, \beta \in \Omega^1(M)$.
\end{exemplo}

\vspace{2mm}

Let $\mathbb{E} \rightarrowtail M$ be a Courant algebroid. For a vector subbundle $L\to S$ of $\mathbb{E}\to M$, its orthogonal complement $L^{\perp}$ with respect to the fiberwise pairing defines a new subbundle of $\mathbb{E}$ supported on $S$. Recall that $(L^{\perp})^{\perp} = L$. We denote by $\Gamma(\mathbb{E}; L)$ the sections of $\mathbb{E}$ which take values in the subbundle $L$ when restricted to $S$, i.e.
$$
\Gamma(\mathbb{E}; L) := \{ e \in \Gamma(\mathbb{E}) \,\, \vert \,\, e\vert_S \in \Gamma(L) \}.
$$
A \textbf{Dirac structure in $\mathbb{E}$ with support} on a submanifold $S\subseteq M$ is a subbundle $L \rightarrow S$ such that 
\begin{enumerate}[left=10pt, itemsep=0.3em]
	\item[(1)] $L\vert_s$ is Lagrangian in $\mathbb{E}\vert_s$ for all $s \in S$, i.e. $L = L^{\perp}$;
	\item[(2)] $L$ is involutive, i.e. if for any $e_1, e_2 \in \Gamma(\mathbb{E};L)$ we have $[\![ e_1 , e_2 ]\!] \in \Gamma(\mathbb{E}; L)$;
	\item[(3)] $L$ is compatible with the anchor, i.e. $\rho(L) \subseteq TS$.
\end{enumerate}
When $S = M$, we refer to $L$  simply as a \textbf{Dirac structure}, and we say that $(\mathbb{E}, L)$ is a \textbf{Manin pair}. A triple $(\mathbb{E}, L_1, L_2)$ where $\mathbb{E}$ is a Courant algebroid and $L_1, L_2$ are transversal Dirac structures (i.e. $\mathbb{E} = L_1 \oplus L_2$) is called a \textbf{Manin triple}.

\begin{obs}\label{rem:diracLA}
	Dirac structures (with full support) in a Courant algebroid $\mathbb{E} \rightarrowtail M$ inherit a natural Lie algebroid structure by the restriction of anchor and bracket. In case the Dirac structure has support on a proper submanifolds, this is no longer the case, see e.g. \cite[Rem. 2.20]{vysoky2020hitchhiker}.
\end{obs}

\begin{exemplo}\label{exgraphdirac}
	In Example \ref{standardCA}, both $TM$ and $T^*M$ are Dirac structures in $\mathbb{T}M$. More generally, if $\pi \in \mathfrak{X}^2(M)$ is a bivector field, then the graph $gr(\pi^{\sharp}) \subseteq \mathbb{T}M$ of the associated map $\pi^{\sharp} : T^*M \rightarrow TM$ is a Dirac structure if and only if $\pi$ is Poisson.
\end{exemplo}

\vspace{1mm}

The following result generalizes \cite[Proposition 7.1]{liu1997manin} and provides an example of Dirac structures (with support) on Courant algebroids, with the proof following a similar approach.

\begin{proposicao}\label{proposicao}
	Let $(A \Rightarrow M, A^* \Rightarrow M)$ be a Lie bialgebroid. Then, the vector bundles $B \rightarrow N$ and $ann(B) \rightarrow N$ are Lie subalgebroids of $A \Rightarrow M$ and $A^* \Rightarrow M$, if and only if, $L := B \oplus ann(B) \rightarrow N$ is a Dirac structure on $A \oplus A^*$ with support on $N$.
\end{proposicao}

\vspace{1mm}

A central result by Liu-Weinstein-Xu \cite{liu1997manin} assert that Lie bialgebroids are equivalent to Manin triples:
\begin{itemize}[left=10pt, itemsep=0.3em]
\item Given a Lie bialgebroid $(A,A^*)$, the direct sum $A\oplus A^*$ has a natural Courant algebroid structure with respect to which $A$ and $A^*$ are Dirac structures (i.e., $(A\oplus A^*, A, A^*)$ is a Manin triple),
\item For a Manin triple $(\mathbb{E}, L_1, L_2)$, the Lie algebroids $L_1$ and $L_2\cong L_1^*$ (identification via the pairing) form a Lie bialgebroid.
\end{itemize}

\vspace{1mm}

\begin{exemplo}\label{exisoCourantr-matrix}
    Given a Lie algebroid $A \Rightarrow M$ and an r-matrix $\Lambda \in \Gamma(\wedge^2A)$, denote by $A^*_{tr}$ the vector bundle $A^* \rightarrow M$ with the trivial Lie algebroid structure and by $A^*_{\Lambda}$ the vector bundle $A^* \rightarrow M$ with the Lie algebroid structure induced by the r-matrix (see Example \ref{exbialgbdr-matrix}). Therefore, $(A \oplus A_{tr}^*, A, gr(\Lambda^\sharp))$ is a Manin triple, and, by the correspondence outlined above, the Courant algebroids $A \oplus A^*_{tr}$ and $A \oplus A^*_{\Lambda} \simeq A \oplus gr(\Lambda^\sharp)$ are isomorphic.
\end{exemplo}


\section{Lie bialgebroids over Lie groupoids}\label{section3}

To pass from Lie bialgebroids to Lie bialgebroid groupoids, we need to replace Lie algebroids by LA-groupoids \cite{mackenzie1992double}, as we briefly recall.

\subsection{VB-groupoids}
A \textbf{VB-groupoid} consists of Lie groupoids $\Gamma \rightrightarrows H$, $G \rightrightarrows M$, and vector bundles $\Gamma \rightarrow G$, $H \rightarrow M$, which form the following diagram
$$
\begin{tikzcd}[
column sep={2em,between origins},
row sep={1.3em,between origins},]
	\Gamma & \rightrightarrows & H  \\
	\downarrow & & \downarrow \\
	G & \rightrightarrows & M,
\end{tikzcd}
$$
with the compatibility that the Lie groupoid structure maps of $\Gamma$ (source, target, multiplication, unit, inverse) are vector bundle maps covering the Lie groupoid structure maps of $G$. In \cite{bursztyn2016vector}, it was shown that this compatibility is equivalent to asking that the scalar multiplication of $\Gamma \rightarrow G$ is by groupoid morphisms covering the scalar multiplication on $H \rightarrow M$.

\vspace{1mm}

Any VB-groupoid has a \textbf{core} vector bundle $C \rightarrow M$, which is defined as 
$$
C = \{ u \in \Gamma \,\, \vert \,\, \mathtt{s}_{\Gamma}(u) = 0_m, p_1(u) = \epsilon_{G}(m), m \in M \},
$$
i.e. $C$ is the pullback by the identity map $\epsilon_{G} : M \rightarrow G$ of the kernel of the source map $\mathtt{s}_{\Gamma} : \Gamma \rightarrow H$, where $p_1 : \Gamma \rightarrow G$ is the vector bundle projection. An important fact, see e.g. \cite[Section 11.2]{mackenzie2005general}, is that for a VB-groupoid $\Gamma$, its dual $\Gamma^* \rightarrow G$ carries a VB-groupoid structure over the dual of the core:
\begin{equation}\label{eq:dual}
\begin{tikzcd}[
column sep={2em,between origins},
row sep={1.3em,between origins},]
\Gamma^* & \rightrightarrows & C^*  \\
\downarrow & & \downarrow \\
G & \rightrightarrows & M.
\end{tikzcd}
\end{equation}

\begin{exemplo}\label{exvbgrpdoverM}
    A VB-groupoid over the unit groupoid, 
    $$
	\begin{tikzcd}[
		column sep={2em,between origins},
		row sep={1.3em,between origins},]
		\Gamma & \rightrightarrows & H  \\
		\downarrow & {\scriptstyle C} & \downarrow \\
		M & \rightrightarrows & M,
	\end{tikzcd}
	$$
    corresponds to a vector bundle map $\partial := \mathtt{t}\vert_C : C \rightarrow H$. In this case, $\Gamma$ has the structure of an action groupoid for the action of the vector bundle $C \rightarrow M$, regarded as a Lie groupoid with respect to fiberwise addition, on $H \rightarrow M$ by $c \cdot h = h + \partial(c)$ (see \cite[\S 3]{bursztyn2021poisson} for more details).
\end{exemplo}

The direct product of VB-groupoids is naturally a VB-groupoid. Also,
if $\Gamma_1\rightrightarrows H_1$ and $\Gamma_2\rightrightarrows H_2$ are VB-groupoids over the same base $G\rightrightarrows M$, then there is a natural VB-groupoid structure on
$\Gamma_1\oplus \Gamma_2 \rightrightarrows H_1\oplus H_2$ over $G\rightrightarrows M$, see e.g. \cite{bursztyn2016vector}.

\vspace{1mm}

Let $\Gamma$ be a VB-groupoid and $\Gamma^*$ its dual VB-groupoid. Consider the following isomorphisms of VB-groupoids:
\begin{equation}\label{eq:phi}
\varphi: (\Gamma^*)^3\to  (\Gamma^*)^3, \quad (\alpha,\beta,\gamma) \mapsto (\alpha,-\beta,-\gamma), 
\end{equation}
and 
\begin{equation}\label{eq:Phi}
\Phi: (\Gamma\oplus \Gamma^*)^3\to \Gamma^3\oplus (\Gamma^*)^3, \quad  (u \oplus \alpha, v \oplus \beta, w \oplus \gamma) \mapsto (u, v, w) \oplus (\alpha, - \beta, - \gamma).
\end{equation}
We will need the following result.

\begin{proposicao}\label{prop:12}
    The following holds:
    \begin{itemize}[left=13pt, itemsep=0.3em]
\item[(a)] $\varphi(gr(m_{\Gamma^*})) = ann(gr(m_{\Gamma})).$
\item[(b)] $\Phi(gr(m_{\Gamma \oplus \Gamma^*})) = gr(m_\Gamma) \oplus ann(gr(m_\Gamma)).$
    \end{itemize}
\end{proposicao}
\begin{proof}
For part (a), we have that
	\begin{eqnarray}
	ann(gr(m_{\Gamma})) &=& \{ (\gamma, - \alpha, - \beta) \,\, \vert \,\, \langle \gamma, m_{\Gamma}(u,v) \rangle - \langle \alpha, u \rangle - \langle \beta, v \rangle  = 0 \} \nonumber \\
	&=& \{ (\gamma, - \alpha, - \beta) \,\, \vert \,\, \langle \gamma, m_{\Gamma}(u,v) \rangle - \langle m_{\Gamma^*}(\alpha, \beta), m_{\Gamma}(u,v) \rangle = 0 \} \nonumber \\
    &=& \{ (m_{\Gamma^*}(\alpha, \beta), -\alpha, -\beta) \, \vert \, \alpha, \beta \,\, \text{composable} \} \nonumber \\
	&=&  \varphi(gr(m_{\Gamma^*})), \nonumber
	\end{eqnarray}
	where in third line we use that, by definition, $\langle m_{\Gamma^*}(\alpha, \beta), m_{\Gamma}(u,v) \rangle = \langle \alpha, u \rangle + \langle \beta, v \rangle$, and in the fourth line we use that, since $\Gamma$ is a VB-algebroid, for all $w \in \Gamma_{gh}$ there exist $u \in \Gamma_g, v \in \Gamma_h$ such that $w = m_{\Gamma}(u,v)$. 

\vspace{1mm}

To verify part (b), note that by definition we have
	\begin{eqnarray}
	&& \Phi(gr(m_{\Gamma \oplus \Gamma^*})) \nonumber \\
	&=&  \{ (m_\Gamma(u,v),u,v) \oplus (m_{\Gamma^*}(\alpha, \beta), -\alpha, -\beta) \, \vert \, u \oplus \alpha \,\, \text{and} \,\, v \oplus \beta \,\, \text{are composable} \}, \nonumber \\
	&=&   \{ (m_\Gamma(u,v),u,v) \, \vert \, u,v \,\, \text{comp.} \} \oplus \{ (m_{\Gamma^*}(\alpha, \beta), -\alpha, -\beta) \, \vert \, \alpha, \beta \,\, \text{comp.} \}, \nonumber \\
	&=& gr(m_\Gamma) \oplus ann(gr(m_\Gamma)), \nonumber
	\end{eqnarray}
	as desired.
\end{proof}

\subsection{Lie bialgebroid groupoids}

An \textbf{LA-groupoid} is a diagram
   $$
	\begin{tikzcd}[
	column sep={2em,between origins},
	row sep={1.3em,between origins},]
	\Gamma & \rightrightarrows & H  \\
	\Downarrow & & \Downarrow \\
	G & \rightrightarrows & M
	\end{tikzcd}
	$$
consisting of a VB-groupoid $\Gamma \rightrightarrows H$ over $G \rightrightarrows M$ together with Lie algebroid structures $\Gamma \Rightarrow G$ and $H \Rightarrow M$, such that the Lie groupoid structure maps are Lie algebroid morphisms (for the multiplication map, we use that $\Gamma \times_H \Gamma$ inherits a Lie algebroid structure over $G \times_M G$). Note that it is equivalent to require only that $gr(m_\Gamma)$ be a Lie subalgebroid: with this, we can prove that the graph of the identity of $\Gamma$ also satisfies this property, and then show that the graphs of the source and target maps do as well.

\begin{examples}
	Tangent bundles of Lie groupoids and cotangent bundles of Poisson groupoids are examples of LA-groupoids.
\end{examples}

\begin{definicao} (\cite{bursztyn2021poisson})
A \textbf{Lie bialgebroid groupoid} is a pair $(\Gamma, \Gamma^*)$ of LA-groupoids which are in duality as VB-groupoids,
$$
\begin{tikzcd}[
column sep={2em,between origins},
row sep={1.3em,between origins},]
\Gamma & \rightrightarrows & H  \\
\Downarrow & & \Downarrow \\
G & \rightrightarrows & M,
\end{tikzcd}
\,\,\,\,\,\,\,\,\,\,\,\,\,\,\,
\begin{tikzcd}[
column sep={2em,between origins},
row sep={1.3em,between origins},]
\Gamma^* & \rightrightarrows & C^*  \\
\Downarrow & & \Downarrow \\
G & \rightrightarrows & M,
\end{tikzcd}
$$
and whose Lie algebroid structures $(\Gamma \Rightarrow G, \Gamma^* \Rightarrow G)$ form a Lie bialgebroid. 
\end{definicao}

\begin{exemplo}
    Given a Poisson groupoid $(G \rightrightarrows M, \pi)$, the tangent and cotangent spaces, $TG \rightrightarrows TM$ and $T^*G \rightrightarrows Lie(G)$, are LA-groupoids in duality, where $Lie(G)$ is the Lie algebroid of $G$. Since $(TG \Rightarrow G, T^*G \Rightarrow G)$ is a Lie bialgebroid, $(TG, T^*G)$ is a Lie bialgebroid groupoid.
\end{exemplo}


\section{Manin triples for Lie bialgebroid groupoids}\label{section4}

In this section, we provide a characterization of Lie bialgebroid groupoids in terms of Manin triples. We start by recalling the notion of Courant algebroids over Lie groupoids from \cite{li2012courant} (see also \cite{cristianthesis} and \cite{mehta2009q}).

\subsection{CA-groupoids}

Let $G \rightrightarrows M$ be a Lie groupoid.

\begin{definicao} (\cite{li2012courant})
A \textbf{CA-groupoid}, or \textbf{multiplicative Courant algebroid}, over $G$ is a VB-groupoid $\mathcal{G} \rightrightarrows H$ such that $\mathcal{G} \rightarrow G$ is a Courant algebroid,
    \begin{eqnarray}\label{cagroupoid}
	\begin{tikzcd}[
		column sep={2em,between origins},
		row sep={1.3em,between origins},]
		\mathcal{G} & \rightrightarrows & H  \\
		\downarrowtail & {\scriptstyle C} & \downarrow \\
		G & \rightrightarrows & M,
	\end{tikzcd}
	\end{eqnarray} 
and the graph of the multiplication of $\mathcal{G}$ is a Dirac structure in $\mathcal{G} \times \overline{\mathcal{G} \times \mathcal{G}}$, with support on $gr(m_{G})$.
\end{definicao}

\begin{exemplo}[Standard and exact CA-groupoids]\label{ex:standardmult}
	If $G \rightrightarrows M$ is a Lie groupoid, then the tangent lift $TG \rightrightarrows TM$ is an LA-groupoid over $G$, with dual  
    $T^*G \rightrightarrows {Lie}(G)^*$ \cite{weinstein1987symplectic}. Hence, $\mathbb{T}G = TG \oplus T^*G$ is naturally a Lie groupoid, where the graph of the multiplication $gr(m_{\mathbb{T}G})$ is a Dirac structure on $\mathbb{T}G \times \overline{\mathbb{T}G \times \mathbb{T}G}$ with support on $gr(m_G)$, so the standard Courant algebroid $\mathbb{T}G$ is a CA-groupoid (see \cite{li2012courant}). More generally, it was shown in \cite[Proposition 2.3]{alvarez2024transitive} that if one twists the Courant bracket of $TG \oplus T^*G \rightarrowtail G$ by a 3-form $H \in \Omega^3(G)$ and appropriately modifies the Lie groupoid structure of $TG \oplus T^*G \rightrightarrows TM \oplus Lie(G)^*$ by a 2-form $\omega \in \Omega^2(G_2)$, this is again a CA-groupoid if the following conditions are satisfied: $dH = 0, \delta \omega = 0$ and $d \omega = \delta H$, where $\delta$ is the simplicial differential from the Bott-Shulman-Stasheff complex. Moreover, every exact CA-groupoid can be identified with one of this type. We denote this CA-groupoid by $\mathbb{T}_H^{\omega}G$.
\end{exemplo}

\begin{definicao}
	Consider a multiplicative Courant algebroid as in \eqref{cagroupoid}. A Dirac structure $L$ in $\mathcal{G}$ is said to be \textbf{multiplicative} if $L$ is a VB-subgroupoid of $\mathcal{G}$: 
    \begin{eqnarray}\label{defmultDirac}
    \begin{tikzcd}[
	column sep={2em,between origins},
	row sep={1.3em,between origins},]
	L & \rightrightarrows & S  \\
	\downarrow & & \downarrow \\
	G & \rightrightarrows & M
\end{tikzcd}
\,\,\,\,\, \subseteq \,\,\,\,\, 
\begin{tikzcd}[
	column sep={2em,between origins},
	row sep={1.3em,between origins},]
	\mathcal{G} & \rightrightarrows & H  \\
	\downarrowtail & & \downarrow \\
	G & \rightrightarrows & M.
\end{tikzcd}
    \end{eqnarray}
    In this case, we say that the pair $(\mathcal{G}, L)$ is a \textbf{multiplicative Manin pair}. If $L_1, L_2 \subseteq \mathcal{G}$ are transverse, i.e. $\mathcal{G} = L_1 \oplus L_2$, and multiplicative Dirac structures, we say that $(\mathcal{G}, L_1, L_2)$ is a \textbf{multiplicative Manin triple}.
\end{definicao}

\begin{proposicao}\label{prop:DiracLAgrp}
    If $L$ is multiplicative Dirac on $\mathcal{G}$ as in \eqref{defmultDirac}, then it is an LA-groupoid with respect to its induced Lie algebroid structure.
\end{proposicao} 
\begin{proof}
    It follows immediately from the fact that $gr(m_L)\subseteq gr(m_\mathcal{G})$, with $gr(m_\mathcal{G})$ a Dirac structure over $gr(m_G)$, and $L$ is fully supported on $G$.
\end{proof}

\vspace{1mm}

The next proposition says that the compatibility between the fiberwise pairing and the Lie groupoid structure on a CA-groupoid can be expressed in terms of a Lie groupoid morphism. A similar result was proven in \cite[Lemma 3.6]{bursztyn2009linear} for skew-symmetric forms, and the same proof holds for symmetric forms.

\begin{proposicao}\label{grpdmorf}
    Suppose that $\mathcal{G}$ is a CA-groupoid as in \eqref{cagroupoid}. Then the map $(\mathcal{G} \rightrightarrows H) \rightarrow (\mathcal{G}^* \rightrightarrows C^*)$ induced by the fiberwise pairing is a Lie groupoid morphism.
\end{proposicao}

\subsection{Manin triples for Lie bialgebroid groupoids}

In what follows, given a Lie algebroid $A\Rightarrow M$ with anchor $\rho_A$ and bracket $[\cdot \,,\cdot]_A$, we denote by $A^{op}$ the opposite Lie algebroid, defined by the same vector bundle $A\to M$ with anchor and bracket given by $-\rho_A$ and $-[\cdot \,,\cdot]_A$, so that $-Id: A\to A^{op}$ is a Lie algebroid isomorphism. Note that for Manin triples $(\mathbb{E}, A, B)$ and $(\overline{\mathbb{E}}, A, B)$, the Lie algebroid structures on $A^*$ coming from its identification with $ B$ via the respective pairings are opposite to one another.

\vspace{1mm}

Given a Lie bialgebroid $(\Gamma,\Gamma^*)$, consider the Courant algebroid $\Gamma \oplus \Gamma^*$ given by its Drinfeld double, and the product Courant algebroid
$$
\mathbb{E} = (\Gamma \oplus \Gamma^*) \times \overline{(\Gamma \oplus \Gamma^*) \times (\Gamma \oplus \Gamma^*)}.
$$
Then  $A = (\Gamma \oplus 0)^3 \simeq \Gamma^3$ and $B = (0 \oplus \Gamma^*)^3 \simeq (\Gamma^*)^3$ are Dirac structures in $\mathbb{E}$, so that $(\mathbb{E}, A, B)$ is a Manin triple, and the identification of the Courant algebroid $\mathbb{E}$ with the double $A\oplus A^*$ induced by the pairing is given by the map \eqref{eq:Phi},
\begin{eqnarray}
\Phi : (\Gamma \oplus \Gamma^*) \times \overline{(\Gamma \oplus \Gamma^*) \times (\Gamma \oplus \Gamma^*)} & \longrightarrow & (\Gamma \times \Gamma \times \Gamma) \oplus (\Gamma^* \times (\Gamma^*)^{op} \times (\Gamma^*)^{op}) \nonumber \\
(u \oplus \alpha, v \oplus \beta, w \oplus \gamma) & \longmapsto & (u, v, w) \oplus (\alpha, - \beta, - \gamma). \nonumber
\end{eqnarray}
Note also that the map \eqref{eq:phi} defines a  Lie algebroid isomorphism:  
\begin{eqnarray}
\varphi: \Gamma^* \times \Gamma^* \times \Gamma^* &\longrightarrow& \Gamma^* \times (\Gamma^*)^{op} \times (\Gamma^*)^{op} \nonumber \\
(\alpha, \beta, \gamma) &\longmapsto& (\alpha, -\beta, -\gamma). \nonumber
\end{eqnarray}

\vspace{2mm}

Now we are able to prove our main Theorem.

\begin{teorema}\label{teorema1}
{\em The Drinfeld double of a Lie bialgebroid groupoid $(\Gamma, \Gamma^*)$ is a CA-groupoid, so that $(\Gamma \oplus \Gamma^*, \Gamma, \Gamma^*)$ is a multiplicative Manin triple. Conversely, given a multiplicative Manin triple $(\mathcal{G}, L_1, L_2)$, $(L_1, L_2)$ is a Lie bialgebroid groupoid, where $L_2$ is identified with $L_1^*$ via the pairing of $\mathcal{G}$.} 
\end{teorema}
\begin{proof}
    Let $(\Gamma, \Gamma^*)$ be a Lie bialgebroid groupoid. Then $\Gamma\oplus \Gamma^*$ is a Courant algebroid which is also a VB-groupoid,
\begin{equation}\label{eq:CA}
\begin{tikzcd}[
	column sep={2.5em,between origins},
	row sep={1.3em,between origins},]
	\Gamma \oplus \Gamma^* & \rightrightarrows & H \oplus C^*  \\
	\downarrowtail & & \downarrow \\
	G & \rightrightarrows & M.
\end{tikzcd}
\end{equation}
To check that this is a CA-groupoid, we need to show that $gr(m_{\Gamma \oplus \Gamma^*})$ is a Dirac structure in 
    $
    (\Gamma \oplus \Gamma^*) \times \overline{((\Gamma \oplus \Gamma^*) \times (\Gamma \oplus \Gamma^*))}
    $ 
with support on $gr(m_G)$. Using the identification $\Phi$ above and Prop.~\ref{prop:12} (b), this is equivalent to showing that $gr(m_{\Gamma}) \oplus ann(gr(m_{\Gamma}))$ is Dirac on
	$
	(\Gamma \times \Gamma \times \Gamma) \oplus (\Gamma^* \times (\Gamma^*)^{op} \times (\Gamma^*)^{op}).
	$
Since $\Gamma$ and $\Gamma^*$ are LA-groupoids, $m_{\Gamma}$ and $m_{\Gamma^*}$ are Lie algebroid morphisms, and then $gr(m_{\Gamma})$ and $gr(m_{\Gamma^*})$ are Lie subalgebroids over $gr(m_G)$. Using $\varphi$ above and Prop.~\ref{prop:12} (a), it follows that $ann(gr(m_{\Gamma}))$ is also a Lie subalgebroid over $gr(m_G)$. By Prop. \ref{proposicao}, $gr(m_{\Gamma}) \oplus ann(gr(m_{\Gamma}))$ is Dirac on $(\Gamma \times \Gamma \times \Gamma) \oplus (\Gamma^* \times (\Gamma^*)^{op} \times (\Gamma^*)^{op})$, so \eqref{eq:CA} is a CA-groupoid. Finally note that $\Gamma$ and $\Gamma^*$ sit in $\Gamma\oplus \Gamma^*$ as Dirac structures and VB-subgroupoids, so 
$(\Gamma \oplus \Gamma^*, \Gamma, \Gamma^*)$ is a multiplicative Manin triple. 

\vspace{1mm}

Conversely, let $(\mathcal{G}, L_1, L_2)$ be a multiplicative Manin triple. By 
Prop. \ref{prop:DiracLAgrp}, $L_1$ and $L_2$ inherit the structure of LA-groupoids.
By Prop.~\ref{grpdmorf}, we have an identification of  $L_2$ with $L_1^*$ as VB-groupoids via the pairing, and under such identification $(L_1 \Rightarrow G, L_1^* \Rightarrow G)$ is a Lie bialgebroid, hence $(L_1,L_1^*)$ is a Lie bialgebroid groupoid. 
\end{proof}

\begin{exemplo}\label{exCAgrpdtangcotg}
	Given a Poisson groupoid $(G \rightrightarrows M, \pi)$, the tangent and cotangent spaces are LA-groupoids in duality, and $(TG \Rightarrow G, T^*G \Rightarrow G)$ form a Lie bialgebroid, then its double is a CA-groupoid:
	$$
	\begin{tikzcd}[
		column sep={2em,between origins},
		row sep={1.3em,between origins},]
		TG & \rightrightarrows & TM \\
		\Downarrow & & \Downarrow \\
		G & \rightrightarrows & M
	\end{tikzcd}
	\,\, \oplus \,\,
	\begin{tikzcd}[
		column sep={2em,between origins},
		row sep={1.3em,between origins},]
		T^*G & \rightrightarrows & \,\,\,\,\,\,\,\, Lie(G)^*  \\
		\Downarrow & & \,\,\, \Downarrow \\
		G & \rightrightarrows & \,\,\, M
	\end{tikzcd} \,\, = \,\,
     \begin{tikzcd}[
			column sep={2em,between origins},
			row sep={1.3em,between origins},]
			TG \oplus T^*G \,\,\,\,\,\,\,\,\,\,\,\,\,\,\,\,\, & \rightrightarrows & \,\,\,\,\,\,\,\,\,\,\,\,\,\,\,\,\,\,\,\,\,\,\,\,\, TM \oplus Lie(G)^*  \\
			\downarrowtail \,\,\,\,\,\,\,\,\,\,\,\,\,\,\,\,\,\,\, & & \,\,\,\,\,\,\,\,\,\,\,\,\,\,\,\,\,\,\,\,\,\, \downarrow \\
			G \,\,\,\,\,\,\,\,\,\,\,\,\,\,\,\,\, & \rightrightarrows & \,\,\,\,\,\,\,\,\,\,\,\,\,\,\,\,\,\,\,\,\,\, M \nonumber
		\end{tikzcd},
	$$
    where $Lie(G)$ is the Lie algebroid of $G$. This Courant algebroid arising as the double is isomorphic to the standard one in Example \ref{ex:standardmult}. 
\end{exemplo}

\begin{exemplo} 
  Let $\mathbb{T}_HG$ denote the exact CA-groupoid of Example~\ref{ex:standardmult} with $\omega=0$. If $H = d \eta$, for some $\eta \in \Omega^2(G)$ multiplicative, then the CA-groupoid $\mathbb{T}_HG$ is given by the double of a Lie bialgebroid groupoid. Indeed, $L:= gr(\eta^\sharp)$ is a multiplicative Dirac structure  and  $(\mathbb{T}_HG, T^*G, L)$ is a multiplicative Manin triple. 
\end{exemplo}

\begin{exemplo}\label{expairgrpdglob}
    Let $(A \Rightarrow M, A^* \Rightarrow M)$ be a Lie bialgebroid and consider the Lie bialgebroid groupoid $(A \times A, A^* \times A^*)$ by equipping it with the structure of the pair groupoid. By Theorem \ref{teorema1}, its double is a CA-groupoid as follows:
    $$
	\begin{tikzcd}[
		column sep={2em,between origins},
		row sep={1.3em,between origins},]
		A \times A \,\,\,\,\,\,\,\,\,\,\, & \rightrightarrows & A \\
		\Downarrow \,\,\,\,\,\,\,\,\,\,\, & & \Downarrow \\
		M {\times} M \,\,\,\,\,\,\,\,\,\,\, & \rightrightarrows & M
	\end{tikzcd}
	\,\,\, \oplus \,\,\,
	\begin{tikzcd}[
		column sep={2em,between origins},
		row sep={1.3em,between origins},]
		A^* \times A^* \,\,\,\,\,\,\,\,\,\,\, & \rightrightarrows & A^*  \\
		\Downarrow \,\,\,\,\,\,\,\,\,\,\, & &  \Downarrow \\
		M {\times} M \,\,\,\,\,\,\,\,\,\,\, & \rightrightarrows &  M
	\end{tikzcd}
    \quad = \quad
    \begin{tikzcd}[
		column sep={2em,between origins},
		row sep={1.3em,between origins},]
		\mathbb{E} \times \mathbb{E} \,\,\,\,\,\,\,\,\,\,\, & \rightrightarrows & \,\,\,\,\,\,\,\,\,\, A \oplus A^* \\
		\downarrowtail \,\,\,\,\,\,\,\,\,\,\, & & \,\,\,\,\,\,\,\,\,\, \downarrow \\
		M {\times} M \,\,\,\,\,\,\,\,\,\,\, &  \rightrightarrows & \,\,\,\,\,\,\,\,\,\, M
	\end{tikzcd},
	$$
    where $\mathbb{E} := A \oplus A^* \rightarrowtail M$ is the Courant algebroid given by the double.
\end{exemplo}

\begin{exemplo}
    Given a CA-groupoid $\mathcal{G} \rightrightarrows H$ over $G \rightrightarrows M$, and $\Phi : (K \rightrightarrows N) \rightarrow (G \rightrightarrows M)$ a Lie groupoid morphism such that $\Phi$ is transverse to the anchor map $\rho : \mathcal{G} \rightarrow TG$, i.e. $im(d \Phi) + im(\rho) = TG$, then the pullback Courant algebroid $\Phi^{!}\mathcal{G} \rightrightarrows \Phi^*H$ inherits the structure of a CA-groupoid over $K \rightrightarrows N$ (see \cite[\S 2]{li2014dirac}). If the CA-groupoid $\mathcal{G}$ is given by the double of a Lie bialgebroid groupoid $(A, A^*)$ such that the Lie algebroids $A \Rightarrow G$ and $A^* \Rightarrow G$ are transitive, then $L_1 := (A \oplus 0) \times (TK \oplus 0)$ and $L_2 := (0 \oplus A^*) \times (0 \oplus  T^*K)$ are transverse and multiplicative Dirac structures on $(A \oplus A^*) \times \mathbb{T}K$, and then, by \cite[Proposition 2.15]{li2014dirac}, $\frac{L_1 \cap C}{L_1 \cap C^{\perp}}$ and $\frac{L_2 \cap C}{L_2 \cap C^{\perp}}$ are transverse and multiplicative Dirac structures on $\Phi^{!}(A \oplus A^*)$. Therefore, by Theorem \ref{teorema1}, $\big(\frac{L_1 \cap C}{L_1 \cap C^{\perp}}, \frac{L_2 \cap C}{L_2 \cap C^{\perp}}\big)$ is a Lie bialgebroid groupoid. This approach allows the construction of new Lie bialgebroid groupoids from previous ones.
\end{exemplo}


\section{Application: Co-quadratic Lie algebroids revisited}\label{section5}

The notion of co-quadratic Lie algebroid  and their Dirac structures were introduced in \cite{lang2021lie} to provide a Manin triple framework for Lie bialgebroid crossed modules.
In this section, we will show that these can be cast as Manin triples over the unit groupoid
and that \cite[Theorem 3.7]{lang2021lie}, which states that Lie bialgebroid crossed modules correspond to co-quadratic Manin triples, follows as a particular case of Theorem \ref{teorema1}.

\begin{definicao}
    A \textbf{co-quadratic Lie algebroid} is a Lie algebroid $K \Rightarrow M$ together with a symmetric, bilinear 2-form on the dual vector space:
    $$
    \lfloor \,\, , \, \rfloor : \Gamma(K^*) \times \Gamma(K^*) \rightarrow C^{\infty}(M),
    $$
    which is $K$-invariant, i.e., for all $k \in \Gamma(K)$, $\gamma, \gamma' \in \Gamma(K^*)$,
    $$
    \rho(k)\lfloor \gamma, \gamma' \rfloor = \lfloor \mathcal{L}_k \gamma , \gamma' \rfloor + \lfloor \gamma , \mathcal{L}_k \gamma' \rfloor,
    $$
    where $\rho$ is the anchor of $K$. A \textbf{Dirac structure} on $K$ is a Lie subalgebroid $D \Rightarrow M$ of $K$ such that the null space
    $$
    D^0 = \{ \gamma \in K^*_p \, \vert \, p \in M, \langle \gamma, g \rangle = 0, \, \forall d \in D_p \}
    $$
    is isotropic with respect to $\lfloor \,\, , \, \rfloor$. We call $(K, P, Q)$ a \textbf{co-quadratic Manin triple} if $K$ is a co-quadratic Lie algebroid and $P,Q$ are two Dirac structures on $K$ which are transversal to each other.
\end{definicao}

Although the objects defined above may initially appear unrelated to double structures, the following result demonstrates otherwise.

\begin{teorema}\label{teorema_co-quadLA}
{\em     The following statements hold:}
    \begin{enumerate}[left=10pt, itemsep=0.3em]
        \item[(1)] {\em Co-quadratic Lie algebroids are in correspondence with CA-groupoids over the unit groupoid $M \rightrightarrows M$.}
        \item[(2)] {\em Dirac structures on co-quadratic Lie algebroids are in correspondence with multiplicative Dirac structures over the unit groupoid $M \rightrightarrows M$.}
        \item[(3)] {\em Co-quadratic Manin triples are in correspondence with multiplicative Manin triples over the unit groupoid $M \rightrightarrows M$.}
    \end{enumerate}
\end{teorema}

\vspace{1mm}

In the following, we will prove the previous theorem by explaining how such corres\-pondences are obtained. We begin by characterizing CA-groupoids over the unit groupoid $M \rightrightarrows M$. Note that if ($K \Rightarrow M$, $\rho_K$, $[ \,\, , \, ]_K$) is a Lie algebroid and $\partial : K^* \rightarrow K$ is a symmetric $K$-invariant linear map, we can construct a CA-groupoid structure by taking the Courant algebroid $K^* \oplus K \rightarrowtail M$ defined by the following pairing, anchor and bracket:
    \begin{itemize}[left=10pt, itemsep=0.3em]
        \item $( \gamma \oplus k, \gamma' \oplus k' ) = \frac{1}{2} (\langle \gamma , k' \rangle + \langle \gamma' , k \rangle + \langle \gamma , \partial \gamma' \rangle)$, 
        \item $\rho(\gamma \oplus k) = \rho_K(k)$
        \item $[\gamma \oplus k, \gamma' \oplus k'] = \mathcal{L}_k \gamma' - \mathcal{L}_{k'} \gamma + \mathcal{L}_{\partial \gamma} \gamma' + 2\rho^*d(\gamma, k') \oplus [k,k']_K$, 
    \end{itemize}
for $k, k' \in \Gamma(K)$, $\gamma, \gamma' \in \Gamma(K^*)$, and the Lie groupoid structure $K^* \oplus K \rightrightarrows K$ with the following structure maps:
    \begin{itemize}[left=10pt, itemsep=0.3em]
        \item $\mathtt{s}(\gamma \oplus k) = k$,
        \item $\mathtt{t}(\gamma \oplus k) = k + \partial \gamma$,
        \item $m((\gamma, k + \partial \gamma'),(\gamma', k)) = \gamma + \gamma' \oplus k$,
    \end{itemize}
for $k \in K$, $\gamma, \gamma' \in K^*$. The next result tells us that all CA-groupoids over the unit groupoid $M \rightrightarrows M$ are of this type.

\vspace{1mm}

\begin{proposicao}\label{caracCAgrpdoverM}
{\em Every CA-groupoid over the unit groupoid $M \rightrightarrows M$ as follows
    \begin{eqnarray}
	\begin{tikzcd}[
		column sep={2em,between origins},
		row sep={1.3em,between origins},]
		\mathcal{G} & \rightrightarrows & K  \\
		\downarrowtail & {\scriptstyle C} & \downarrow \\
		M & \rightrightarrows & M \nonumber
	\end{tikzcd}
	\end{eqnarray}
is equivalent to the following data: a Lie algebroid $K \Rightarrow M$ and a $K$-invariant symmetric linear map $\partial : K^* \rightarrow K$, i.e. $\partial$ satisfies the following identities:}
    \begin{enumerate}[left=10pt, itemsep=0.3em]
        \item[(a)] {\em $\partial=\partial^*$,}
        \item[(b)] {\em $\rho(k) \langle \partial \gamma , \gamma' \rangle = \langle \mathcal{L}_k \gamma , \partial \gamma' \rangle + \langle \mathcal{L}_k \gamma' , \partial \gamma \rangle$, for $k \in \Gamma(K)$, $\gamma \in \Gamma(K^*)$.}
    \end{enumerate}
\end{proposicao}
\begin{proof}
    As recalled in Example \ref{exvbgrpdoverM}, the structure of the VB-groupoid $\mathcal{G}$ is encoded on two vector bundles $K \rightarrow M$ and $C \rightarrow M$ (its side and core) and a linear map $\partial : C \rightarrow K$. By \cite[Proposition 2.3]{li2014dirac}, $K \rightarrow M$ is a Dirac structure on $\mathcal{G} \rightarrowtail M$, and then $K \Rightarrow M$ is a Lie algebroid. Let us highlight the compatibilities of $\mathcal{G}$ being a CA-groupoid:
    \begin{enumerate}[left=8pt, itemsep=0.3em]
        \item[(1)] $gr(m_{\mathcal{G}})$ is isotropic in $\mathcal{G} \times \overline{\mathcal{G} \times \mathcal{G}}$ if and only if 
        \begin{enumerate}[left=7pt, itemsep=0.3em]
            \item[(1.1)] $K$ is isotropic,
            \item[(1.2)] $(c,c') = (c, \partial c')$,
        \end{enumerate}
        \item[(2)] $gr(m_{\mathcal{G}})$ is involutive in $\mathcal{G} \times \overline{\mathcal{G} \times \mathcal{G}}$ if and only if 
        \begin{enumerate}[left=7pt, itemsep=0.3em]
            \item[(2.1)] $[c, c'] = [\partial c, c']$,
            \item[(2.2)] $[k, \partial c] = \partial [k, c]$,
            \item[(2.3)] $[\partial c, k] = \partial [c, k]$,
        \end{enumerate}
        \item[(3)] $gr(m_{\mathcal{G}})$ is compatible with the anchor $\rho$ of $\mathcal{G}$ if and only if $im \partial \subseteq ker \rho$,
    \end{enumerate}
    for $c, c' \in \Gamma(C)$, $k \in \Gamma(K)$. By $(1.2)$ and since $K$ is Lagrangian with respect to the non-degenerate metric of $\mathcal{G}$, we have the identification $C \simeq K^*$, and, under this identification, we have $\langle \gamma , k \rangle = 2 (\gamma , k)$ for $\gamma \in \Gamma(K^*)$ and $k \in \Gamma(K)$. Since $( \,\, , \,)$ is symmetric, $\partial : K^* \rightarrow K$ is symmetric, so $(a)$ is satisfied. By $(2.2)$ and $(1.2)$ we have that
    \begin{eqnarray}
        \langle \mathcal{L}_k \gamma, \gamma' \rangle &:=& \rho(k) \langle \gamma , \partial \gamma' \rangle - \langle \gamma , [k, \partial \gamma'] \rangle \nonumber \\
        &=& 2 \rho(k) ( \gamma , \partial \gamma' ) - 2 ( \gamma , \partial [k, \gamma'] ) \nonumber \\
        &=& 2 \rho(k) ( \gamma , \gamma' ) - 2 ( \gamma , [k, \gamma'] ). \nonumber
    \end{eqnarray}
    Since $\mathcal{G}$ is a Courant algebroid, we get that $\rho(k)(\gamma, \gamma') = ([k, \gamma], \gamma') + (\gamma, [k, \gamma'])$, and then
    \begin{eqnarray}
        \langle \mathcal{L}_k \gamma, \gamma' \rangle + \langle \mathcal{L}_k \gamma', \gamma \rangle &=& 4 \rho(k) ( \gamma , \gamma' ) - 2 ( \gamma , [k, \gamma'] ) - 2 ( \gamma' , [k, \gamma] ) \nonumber \\
        &=& 2 \rho(k) ( \gamma , \gamma' ) = 2 \rho(k) ( \partial \gamma , \gamma' ) \nonumber \\
        &=& \rho(k) \langle \partial \gamma , \gamma' \rangle, \nonumber
    \end{eqnarray}
    which shows that $(b)$ is satisfied.
\end{proof}

\vspace{1mm}

By taking $\lfloor \gamma , \gamma' \rfloor := \langle \partial \gamma, \gamma' \rangle$ on Proposition \ref{caracCAgrpdoverM}, for $\gamma, \gamma' \in \Gamma(K^*)$, we see immediately that CA-groupoids over $M \rightrightarrows M$ correspond exactly to co-quadratic Lie algebroids. Next, we will relate multiplicative Dirac structures over $M \rightrightarrows M$ to Dirac structures on co-quadratic Lie algebroids:

\begin{proposicao}\label{multMT=coquadMT}
    {\em Multiplicative Manin triples over the unit groupoid $M \rightrightarrows M$ are in correspondence with co-quadratic Manin triples.}
\end{proposicao}
\begin{proof}
    By Proposition \ref{caracCAgrpdoverM}, we already know how to relate CA-groupoids and co-quadratic Lie algebroids, and it is clear that transverse subbundles on one of these objects will correspond to transverse subbundles on the other one, so we only need to show that their Dirac structures are in correspondence.

    \vspace{1mm}

    Let $D$ be a Dirac structure on a co-quadratic Lie algebroid $K \Rightarrow M$, and consider $K^* \oplus K$ its associated CA-groupoid. Note that, since $D^0$ is isotropic with respect to $\lfloor \,\, , \, \rfloor$, we have that $\partial \vert_{D^0} \subseteq (D^0)^0 = D$. Using this last property, it is straightforward to observe that $D^0 \oplus D$ is isotropic and involutive on $K^* \oplus K \rightarrowtail M$, and also that $D^0 \oplus D$ is a Lie subgroupoid of $K^* \oplus K \rightrightarrows K$. Therefore, $D^0 \oplus D$ is a multiplicative Dirac structure on $K^* \oplus K$. 

    \vspace{1mm}

    Now, let $L$ be a multiplicative Dirac structure on a CA-groupoid $K^* \oplus K$ over $M \rightrightarrows M$, and define $D =: \mathtt{s}(L) = pr_K(K)$. Since $L$ is involutive and the bracket of $K^* \oplus K$ restricts to $K$ is just the Lie bracket of the Lie algebroid $K \Rightarrow M$, it is clear that $D$ is a Lie subalgebroid of $K \Rightarrow M$. Since $L \rightrightarrows D$ is a Lie subgroupoid of $K^* \oplus K \rightrightarrows K$, we have that $\mathtt{t}(\gamma \oplus k) := k + \partial \gamma \in D$ for $\gamma \oplus k \in L$, and then $\partial \gamma \in D$ for all $\gamma \oplus k \in L$. Therefore, $(\gamma, \gamma' \oplus k') = 0$ for all $\gamma \in D^0$ and $\gamma' \oplus k' \in L$, so $D^0 \subseteq L^{\perp} = L$, and then, since $L$ is isotropic, we get $\lfloor \gamma , \gamma' \rfloor := 2(\partial \gamma, \gamma') = 2(\gamma, \gamma') = 0$, for all $\gamma, \gamma' \in D^0$. It follows that $D^0$ is isotropic with respect to $\lfloor \,\, , \, \rfloor := 2 (\,\, , \,)$, and then $D$ is a Dirac structure on the co-quadratic Lie algebroid $K$.
\end{proof}

Now we will show how Theorem \ref{teorema1} recovers \cite[Theorem 3.7]{lang2021lie}, which states that Lie bialgebroid crossed modules correspond to co-quadratic Manin triples. The concept of Lie algebroid crossed modules \cite{androulidakis2005crossed} is a natural generalization of Lie algebra crossed modules, and is defined as follows: a \textbf{Lie algebroid crossed module} over a manifold $M$ is a quadruple $(\theta, \phi, A, \bullet)$, shortly denoted by $(\theta \overset{\phi}{\to} A)$, where $\theta \rightarrow M$ is a Lie algebra bundle, $A \Rightarrow M$ is a Lie algebroid, $\phi : \theta \rightarrow A$ is a Lie algebroid morphism and $\bullet : A \rightarrow \mathcal{D}(\theta)$ is a representation of $A$ on $\theta$, such that:
    \begin{enumerate}[left=10pt, itemsep=0.3em]
        \item[(1)] $\phi(c) \bullet c' = [c,c']_{\theta}$, for $c, c' \in \Gamma(\theta)$;
        \item[(2)] $\phi(v \bullet c) = [v, \phi(c)]_{A}$, for $v \in \Gamma(A), c \in \Gamma(\theta)$.
    \end{enumerate}
As recalled in Example \ref{exvbgrpdoverM}, a VB-groupoid over the unit groupoid $M \rightrightarrows M$ is encoded on two vector bundles $K \rightarrow M$ and $C \rightarrow M$ (its side and core) and a linear map $\partial : C \rightarrow K$. Along these lines, by considering $\phi = \partial$ and the Lie bracket between sections of $\theta$ and $A$ given by the representation, a Lie algebroid crossed module is an LA-groupoid as follows
$$
\begin{tikzcd}[
	column sep={2em,between origins},
	row sep={1.3em,between origins},]
	\theta \oplus A \,\,\,\,\,\,  & \rightrightarrows & A \\
	\Downarrow \,\,\,\,\,\, & {\scriptstyle \theta} & \Downarrow \\
	M \,\,\,\,\,\, & \rightrightarrows & M
\end{tikzcd}.
$$
From this viewpoint, a \textbf{Lie bialgebroid crossed module} \cite{lang2021lie} is just a Lie bialgebroid groupoid as follows
$$
\begin{tikzcd}[
	column sep={2em,between origins},
	row sep={1.3em,between origins},]
	\theta \oplus A \,\,\,\,\,\,  & \rightrightarrows & A \\
	\Downarrow \,\,\,\,\,\, & {\scriptstyle \theta} & \Downarrow \\
	M \,\,\,\,\,\, & \rightrightarrows & M
\end{tikzcd}
\quad
\begin{tikzcd}[
	column sep={2em,between origins},
	row sep={1.3em,between origins},]
	A^* \oplus \theta^* \,\,\,\,\,\,\,\,  & \rightrightarrows & \theta^* \\
	\Downarrow \,\,\,\,\,\,\,\, & {\scriptstyle A^*} & \Downarrow \\
	M \,\,\,\,\,\,\,\, & \rightrightarrows & M
\end{tikzcd}.
$$
Therefore, by applying Theorem \ref{teorema1} and Theorem \ref{teorema_co-quadLA}, we get that co-quadratic Manin triples are in correspondence with Lie bialgebroid crossed modules. This is precisely the result from \cite[Theorem 3.7]{lang2021lie}, which we recover as a special case of Theorem \ref{teorema1} for CA-groupoids over the unit groupoid.

\begin{exemplo}
    Given a Lie algebroid crossed module $(\theta \overset{\phi}{\to} A)$, we call $\mathbf{r} \in \Gamma(\wedge^2 \theta)$ a \textit{crossed module r-matrix} \cite{lang2021lie} if $X \bullet [\mathbf{r}, \mathbf{r}] = 0$ for all $X \in \Gamma(A)$. In this case, $\mathbf{r}$ is an r-matrix on $\theta$, and it gives rise to a Lie algebroid crossed module structure on $(A^* \overset{\phi^T}{\to} \theta^*)$, with the Lie algebroid structure on $\theta^*$ as in Example \ref{exbialgbdr-matrix}. By using the morphism $\phi : \theta \rightarrow A$, $\mathbf{r} \in \Gamma(\wedge^2 \theta)$ gives rise to a form $\mathbf{r'} \in \Gamma(A \wedge \theta) \oplus \Gamma(\wedge^2 A)$ such that $\mathbf{r} + \mathbf{r'}$ is an r-matrix of the Lie algebroid $\theta \oplus A \Rightarrow M$. It can be proved that the Lie algebroid structure on $A^* \oplus \theta^* \Rightarrow M$ comes from $\mathbf{r} + \mathbf{r'}$, and hence $(\theta \oplus A \Rightarrow M, A^* \oplus \theta^* \Rightarrow M)$ is a Lie bialgebroid. Therefore, the Lie algebroid crossed modules $(\theta \overset{\phi}{\to} A)$ and $(A^* \overset{\phi^T}{\to} \theta^*)$ form a Lie bialgebroid crossed module, which gives rise to a Lie bialgebroid groupoid, as constructed above. Its Drinfeld double, as shown in Theorem \ref{teorema1}, is a CA-groupoid, whose Courant algebroid structure is isomorphic to that one given by Drinfeld double of the Lie bialgebroid $(\theta \oplus A, \theta^* \oplus A^*)$, where the Lie algebroid structure on $\theta^* \oplus A^*$ is the trivial one (see Example \ref{exisoCourantr-matrix}). For more details, see \cite{lang2021lie}.
\end{exemplo}

\vspace{2mm}

Considering the even more particular case in which we consider double structures over a point, a \textbf{Lie 2-algebra} \cite{baez2003higher} is an LA-groupoid over $* \rightrightarrows *$ and a \textbf{Lie 2-bialgebra} \cite{bai2013lie} is a Lie bialgebroid groupoid over $* \rightrightarrows *$ (see \cite[\S 6]{bursztyn2021poisson}). On the other side, a CA-groupoid over a point coincides with the categorification of a quadratic Lie algebra: a \textbf{quadratic Lie 2-algebra} is a Lie 2-algebra
    $$
	\begin{tikzcd}[
		column sep={2em,between origins},
		row sep={1.3em,between origins},]
		\mathcal{G} & \rightrightarrows & \mathfrak{g}_0  \\
		\Downarrow & & \Downarrow \\
		* & \rightrightarrows & *
	\end{tikzcd}
	$$
with a \textit{multiplicative bilinear form} $\langle \,\, ,\, \rangle$ on $\mathcal{G}$, i.e. $gr(m_\mathcal{G})$ is isotropic on $\mathcal{G} \times \overline{\mathcal{G} \times \mathcal{G}}$, such that $(\langle \,\, ,\, \rangle, \mathcal{G} \Rightarrow *)$ is a quadratic Lie algebra. In this setup, the concept of Manin triples of Lie 2-algebras \cite{bai2013lie} coincides with the notion of multiplicative Manin triples. Therefore, the main result of this paper also recovers \cite[Theorem 2.12]{bai2013lie} as a special case.

\appendix

 \bibliographystyle{acm}
 \bibliography{biblio2}

\end{document}